\documentclass[11pt,intlimits,draft,oneside]{amsart}

\usepackage[latin2]{inputenc}
\usepackage{latexsym}
\usepackage{amssymb}
\usepackage{amsthm}
\usepackage{color}

\usepackage{amsmath}

\usepackage{enumerate}

\input{cyracc.def}
\newfont{\cyrfnt}{wncyi10 at 11pt}

\newtheorem{thm}{Theorem}[section]
\newtheorem{prop}[thm]{Proposition}
\newtheorem{lemma}[thm]{Lemma}
\newtheorem{cor}[thm]{Corollary}

\theoremstyle{definition}

\newtheorem{remark}[thm]{Remark}
\newtheorem{remarks}[thm]{Remarks}
\newtheorem{example}[thm]{Example}

\newcommand{\C}{\mathbb{C}}
\newcommand{\N}{\mathbb{N}}

\renewcommand{\Im}{\mathrm{Im}\,}
\renewcommand{\Re}{\mathrm{Re}\,}

\def\rg{\mathop{\mathrm{rg}}}

\def\LLL{\mathcal{L}}

\begin{document}

\title%[Cogenerator and semigroup growth]
%{On cogenerators and semigroup growth}
{The growth of a $C_0$-semigroup characterised by its cogenerator}
%\footnote{}

\author{Tanja Eisner}
\address{Tanja Eisner \newline Mathematisches Institut, Universit\"{a}t T\"{u}bingen
\newline Auf der Morgenstelle 10, D-72076, T\"{u}bingen, Germany}
\email{talo@fa.uni-tuebingen.de}

\author{Hans Zwart} 
\address{Hans Zwart \newline 
Department of Applied Mathematics,
University of Twente \newline P.O. Box 217, 7500 AE Enschede, The Netherlands }
\email{h.j.zwart@math.utwente.nl}


\begin{abstract}
  We characterise contractivity, boundedness and polynomial
  growth for a $C_0$-semi\-group in terms of its cogenerator $V$
  (or the Cayley transform of the generator) or its resolvent.  In
  particular, we extend results of Gomilko and Brenner, Thom\'ee and
  show that polynomial growth of a semigroup implies polynomial
  growth of its cogenerator. As is shown by an example, the
  result is optimal. For analytic semigroups we show that the converse
  holds, i.e., polynomial growth of the cogenerators implies
  polynomial growth of the semigroup.
  In addition, we show by simple examples in $(\C^2,\|\cdot\|_p)$, $p
  \neq 2$, that our results on the characterization of contractivity
  are sharp. These examples also show that the famous
  Foia{\c{s}}--Sz.-Nagy theorem on co\-ge\-ne\-rators of contractive
  $C_0$-semigroups on Hilbert spaces fails in $(\C^2,\|\cdot\|_p)$ for
  $p\neq 2$.
\end{abstract}

\maketitle

\vspace{0.1cm}

%%%%%%%%%%%%%%%%%%%%%%%%%%%%%%%%%%%%%%%%%%%%%%%%%%%%%%%%%%%%%%%%%%%%%%%
%%                                                                   %%
%%             INTRODUCTION                                          %%
%%                                                                   %%
%%%%%%%%%%%%%%%%%%%%%%%%%%%%%%%%%%%%%%%%%%%%%%%%%%%%%%%%%%%%%%%%%%%%%%%

\section{Introduction}

In the theory of $C_0$-semigroups, it is a major task to characterise
the asymptotics of the semigroup $(T(t))_{t\geq 0}$ in terms of its
generator $A$. Since $A$ usually is an unbounded operator, one uses
its resolvent $R(\lambda, A)$ as a family of bounded operators.  On
the other hand, the (negative) Cayley transform of $A$ defined as
$$V:=(A+I)(A-I)^{-1},$$
whenever $1\in\rho(A)$, is a bounded operator and determines
$A$ and hence the semigroup uniquely. This operator is called the
\emph{cogenerator} of the semigroup $T(\cdot)$.  

In 1960, Sz.-Nagy and Foia{\c{s}} characterised cogenerators of
contractive $C_0$-semigroups on Hilbert spaces in the following way.
Here and later we denote by $\LLL(H)$ the space of linear bounded
operators on $H$.
\begin{thm}\emph{(Foia{\c{s}}, Sz.-Nagy \cite{szokefalvi/foias:1960}
    or \cite[Theorem III.8.1]{sznagy/foias})} Let $H$ be a Hilbert
  space and $V\in\LLL(H)$. Then $V$ is the cogenerator of a
  contractive $C_0$-semigroup if and only if $V$ is contractive and
  $1\notin P_\sigma(V)$.
\end{thm}

In Hilbert spaces, not only contractivity is preserved by the
cogenerator. Sz.-Nagy and Foia{\c{s}} showed also that a
$C_0$-semigroup is normal, unitary, self-adjoint, isometric,
completely non-unitary and strongly stable if and only if its
cogenerator is normal, unitary, self-adjoint, isometric, completely
non-unitary and strongly stable, respectively (see Sz.-Nagy,
Foia{\c{s}} \cite[Sections III.8-9]{szokefalvi/foias:1960}). Note that
all these results strongly depend on Hilbert space techniques.

The question whether on Hilbert spaces boundedness of a
$C_0$-semigroup implies power boundedness of its cogenerator is still
open. Gomilko \cite{gomilko:2004} and Guo, Zwart \cite{guo/zwart:2006}
showed that this holds for analytic semigroups and for semigroups such
that the semigroup generated by the inverse of the generator is
bounded as well.  Gomilko \cite{gomilko:2004} proved that boundedness
of a semigroup implies that the powers $V^n$ of its cogenerator grow
at most like $\ln(1+n)$.

For Hilbert spaces Guo and Zwart \cite{guo/zwart:2006} proved that
boundedness of the semigroup implies uniform boundedness of
$V^n(V-I)$. On Banach spaces one can only show that $\|V^n\|\leq
c(1+\sqrt{n})$ for some $c$, see Brenner, Thom\'ee \cite{BrTh79}.
Recently, Piskarev, Zwart \cite{PiZw07} proved that this result (and
even the estimate $\|V^n(V-I)\|\leq c(1+\sqrt{n})$) is optimal.  On
the other side, Gomilko, Zwart, Tomilov \cite{gomilko/zwart/tomilov}
showed that on every $l^p$-space, $1<p<\infty$, $p\neq 2$ there exists
a contraction $V$ with $1\notin P_\sigma(V)$ which is not the
cogenerator of a $C_0$-semigroup.
An analogous example on the space $c_0$ follows from Komatsu \cite[pp.
343--344]{komatsu:1966}, see Section \ref{section:examples}.

In this paper we study the connection between contractivity,
boundedness and polynomial growth of a $C_0$-semigroup on a Banach
space and analogous properties of its cogenerator.

We first characterise cogenerators of contractive and bounded
semigroups in terms of the behaviour of the resolvent of the
cogenerator near the point $1$ (Section
\ref{section:charact-resolvent}).  This is an analogon to the
Hille--Yosida theorem for generators.
Then, in Section \ref{section:v_tau}, we generalise the above Foia{\c{s}}--Sz.-Nagy
theorem to Banach spaces using the cogenerator itself and some naturally related operators. Note that although the proofs  in this section are easy, the presented method seems to be promising.
We also discuss the connection with the inverse of a generator and growth of the corresponding
semigroup (see Zwart \cite{zwart:2007-BCP, zwart:2007-SF} 
Gomilko, Zwart \cite{gomilko/zwart:2007}, Gomilko, Zwart, Tomilov
\cite{gomilko/zwart/tomilov}, de Laubenfels \cite{delaubenfels:1988}
for this aspect). 

In Section \ref{section:examples} we present elementary examples of
non-contractive semigroups with contractive cogenerators and conversely
contractive semigroups with non-contractive cogenerators. In
particular, we show that 
the Foia{\c{s}}--Sz.-Nagy theorem fails for 
semigroups on $(\C^2, \|\cdot\|_p)$ with $p\neq 2$. 

Finally, in Section \ref{section:pol-bdd} we show that polynomial
growth of a $C_0$-semigroup on a Banach space implies polynomial
growth of (the powers of) its cogenerator.  We also show that the
provided growth is the best possible. This generalises the result of
Brenner, Thom\'ee \cite{BrTh79} and extends the result of Gomilko
\cite{gomilko:2004} mentioned above.  Conversely, we prove that for
analytic semigroups polynomial growth of the cogenerator is also
sufficient for polynomial growth of the semigroup.

%\hspace{0.05cm}

%%%%%%%%%%%%%%%%%%%%%%%%%%%%%%%%%%%%%%%%%%%%%%%%%%%%%%%%%%%%%%%%%%%%%%%
%%                                                                   %%
%%                 RESOLVENT  CHARACTERISATIONS                      %%
%%                                                                   %%
%%%%%%%%%%%%%%%%%%%%%%%%%%%%%%%%%%%%%%%%%%%%%%%%%%%%%%%%%%%%%%%%%%%%%%%

\section{Characterisations  via the resolvent}\label{section:charact-resolvent}

In this section we give a resolvent characterisation of cogenerators of bounded and
contractive $C_0$-semi\-gro\-ups on Banach spaces.
%using the growth of the resolvent near the point $1$. 
This can be viewed as an analogue of the Hille--Yosida theorem for
generators. 

We first need the following easy lemma.
\begin{lemma}\label{lemma:resolvent}
  Let $X$ be a Banach space and $V\in\LLL(X)$ such that $1\notin
  P_\sigma(V)$. Then the operator $A:\rg(V-I)\to X$ defined by
  $A:=(V+I)(V-I)^{-1}$ is closed and satisfies the following:
  \begin{enumerate}[1)]
  \item $\rho(A)\setminus\{1\}=\left\{\frac{\mu+1}{\mu-1}, 1\neq \mu \in \rho(V)\right\}$;
  \item For $\lambda\in\rho(A)\setminus\{1\}$ one has
  \begin{equation}
    R(\lambda,A)=\frac{1}{\lambda-1} (I-V)R\left(\frac{\lambda+1}{\lambda-1},V\right). 
  \end{equation}
\end{enumerate}
\end{lemma}
\begin{proof} By $A=I+2(V-I)^{-1}$, $A$ is closed and $1\in \rho (A)$. 
Assertion 1) follows from the spectral mapping theorem for
$(V-I)^{-1}$ (see e.g. Engel, Nagel \cite[Theorem IV.1.13]{engel/nagel:2000}), while assertion 2) follows from 
\begin{eqnarray*}
\lambda I - A &=& (\lambda V - \lambda I -V-I)(V-I)^{-1} \\
            &=& (V(\lambda - 1) - (\lambda +1))(V-I)^{-1}
            = (\lambda -1) \left(\frac{\lambda+1}{\lambda-1} - V\right)(I-V)^{-1}
  \end{eqnarray*}
for every $\lambda\neq 1$.
\end{proof}

We now characterise cogenerators of contractive $C_0$-semigroups on
Banach spaces. The ne\-cessary and sufficient condition uses the behaviour of the resolvent of $V$ near the point $1$. 
\begin{thm}\label{thm:cogen-contr}
Let $X$ be a Banach space and $V\in\LLL(X)$. Then the following
conditions are equivalent.
\begin{enumerate}[(i)]
\item $V$ is the cogenerator of a contraction $C_0$-semigroup on $X$.

\item  $V-I$ is injective and has dense range; $(1,\infty)\in \rho(V)$
  and 
   \begin{equation} 
   \|(I-V)R(\mu,V)\|\leq \frac{2}{\mu+1} \quad \text{for all } \mu>1.
     \end{equation}

\item $V-I$ is injective and has dense range; there exists $\mu_0>1$
  such that $(1,\mu_0)\in\rho(V)$ and 
   \begin{equation} 
   \|(I-V)R(\mu,V)\|\leq  \frac{2}{\mu+1} \quad \text{for all } \mu\in(1,\mu_0).
     \end{equation}
\end{enumerate}
\end{thm}
\begin{proof}
We first note that injectivity and dense range of the operator $V-I$
is necessary for every cogenerator $V$. 
Assume now $V-I$ to be injective and to have dense range.    

Define $A:=(V+I)(V-I)^{-1}$ which is densely defined. By the
Hille--Yosida theorem $A$ generates a contraction semigroup if and only
if some $(\lambda_0,\infty)\subset\rho(A)$ and
$\|R(\lambda,A)\|\leq\frac{1}{\lambda}$ for all
$\lambda>\lambda_0\geq 0$. 
%new 
Note that for
$\mu:=\frac{\lambda+1}{\lambda-1}$, $\lambda>\lambda_0>1$ holds if and
only if $1<\mu<\mu_0$ for $\mu_0:=\frac{\lambda_0+1}{\lambda_0-1}$. 
%new
Moreover, by Lemma \ref{lemma:resolvent} we have for $1<\mu\in\rho(V)$ that
%new
$1<\lambda:=\frac{\mu+1}{\mu-1}\in\rho(A)$ and 
%new
\begin{equation}\label{eq:res-gen-to-cogen}
  \lambda R(\lambda,A)=\frac{\lambda}{\lambda-1} (I-V)R(\mu,V)=\frac{\mu+1}{2} (I-V)R(\mu,V).
\end{equation}
This proves the equivalence of (i) and (iii). Using the same
arguments one shows (i)$\Leftrightarrow$(ii).
\end{proof}

Analogously, one proves the following resolvent characterisation of
cogenerators of bounded $C_0$-semi\-groups on Banach spaces.
\begin{thm}\label{thm:cogen-bdd}
Let $X$ be a Banach space, $V\in\LLL(X)$ and $M\geq 1$. Then the following
conditions are equivalent.
\begin{enumerate}[(i)]
\item $V$ is the cogenerator of a $C_0$-semigroup $(T(t))_{t\geq 0}$
  on $X$ satisfying $\|T(t)\|\leq M$ for all $t\geq 0$.
\item  $V-I$ is injective and has dense range; $(1,\infty)\in \rho(V)$
  and 
   \begin{equation} 
   \|\left[(I-V)R(\mu,V)\right]^n\|\leq \frac{2^n M}{(\mu+1)^n} \quad
   \text{for all } \mu>1,\ n\in\N.
     \end{equation}

\item $V-I$ is injective and has dense range; there exists $\mu_0>1$
  such that $(1,\mu_0)\in\rho(V)$ and 
   \begin{equation} 
   \|\left[(I-V)R(\mu,V)\right]^n\|\leq  \frac{2^n M}{(\mu+1)^n} \quad \text{for all }
   \mu\in(1,\mu_0), \ n\in\N.
     \end{equation}
\end{enumerate}
\end{thm}
\clearpage

%\pagebreak[1]
\begin{remarks}
\nopagebreak\mbox{}\nopagebreak
\begin{enumerate}
\item One can replace the intervals $(1,\infty)$ and
$(1,\mu_0)$ in Theorems \ref{thm:cogen-contr} and \ref{thm:cogen-bdd}
by a sequence  of numbers $\mu_n>1$ converging to $1$. 
Indeed, from the proof of the Hille--Yosida theorem follows that it
suffices to check the condition on the resolvent only for a sequence
$\{\lambda_n\}_{n=0}^\infty\subset (0,\infty)$ converging to infinity. 
\item
 Note that $(I-V)R(\mu,V)=I-(\mu-1)R(\mu,V)$. Therefore one can replace the estimates
in (ii) and (iii) in Theorems  \ref{thm:cogen-contr} and
\ref{thm:cogen-bdd} by 
$$\|I-(\mu-1)R(\mu,V)\|\leq \frac{2}{\mu+1}$$ 
or
$$\|\left[I-(\mu-1)R(\mu,V)\right]^n\|\leq \frac{2^n M}{(\mu+1)^n},$$
respectively.
\item
Conditions (ii) and (iii) in Theorem \ref{thm:cogen-bdd} involve
all powers of the resolvent of $V$ and are therefore difficult
to check. Therefore it is desirable to find a simpler (sufficient)
condition on a bounded operator $V$ to be the cogenerator of a bounded $C_0$-semigroup.
\end{enumerate}
\end{remarks}

\vspace{0.05cm}

%%%%%%%%%%%%%%%%%%%%%%%%%%%%%%%%%%%%%%%%%%%%%%%%%%%%%%%%%%%%%%%%%%%%%%%
%%                                                                   %%
%%                  MORE  CHARACTERISATIONS                          %%
%%                                                                   %%
%%%%%%%%%%%%%%%%%%%%%%%%%%%%%%%%%%%%%%%%%%%%%%%%%%%%%%%%%%%%%%%%%%%%%%%

\section{Characterisation via cogenerators of the rescaled semigroups}\label{section:v_tau}

In this section we study the direct connection between the semigroup
and its cogenerator without using the resolvent.  The simplest example
of such a connection is the Foia{\c{s}}--Sz.-Nagy theorem.  However,
the analogous assertion does not hold on Banach spaces (see Section
\ref{section:examples} for elementary examples).  In our approach we
consider the cogenerators of the rescaled semigroups.

We begin with the following observation. If $A$ generates a
contractive or bounded $C_0$-semigroup, then also all 
operators $\tau A$ do for $\tau>0$.
However, as we will see in Section \ref{section:examples}, it is not always true
that the operators 
\begin{equation}
V_\tau:=(\tau A+I)(\tau A-I)^{-1},\quad \tau>0, 
\end{equation}
remain contractive when $V$ is. 

The following proposition characterises generators of contraction
semigroups in terms of $V_\tau$. 
\begin{prop}\label{prop:contr-sgrs-via-v-tau}
Let $A$ be a densely defined operator on a Banach space $X$. Then the following
assertions are equivalent.
\begin{enumerate}[(i)]
\item $A$ generates a contraction $C_0$-semigroup on $X$.
\item $(0,\infty)\subset \rho(A)$ and the operators $V_\tau$ satisfy 
   \begin{equation*} 
   \|V_\tau-I\|\leq 2 \quad \text{for all } \tau>0.
     \end{equation*}
\item There exists $\tau_0>0$ such that
  $(\frac{1}{\tau_0},\infty)\subset\rho(A)$ and the operators $V_\tau$ satisfy
   \begin{equation*} 
   \|V_\tau-I\|\leq 2 \quad \text{for all } 0<\tau<\tau_0.
     \end{equation*}
\end{enumerate}
\end{prop}
\begin{proof}
By the formula 
\begin{equation*} 
   V_\tau = (A+tI)(A-tI)^{-1}=I-2tR(t,A)
    \end{equation*}
for $t:=\frac{1}{\tau}$ we immediately obtain that 
\begin{equation}\label{eq:resolvent-via-v-tau} 
   tR(t,A)=\frac{I-V_\tau}{2}.
    \end{equation}
Then the proposition follows from the Hille--Yosida theorem.
\end{proof}
We now obtain the Foia{\c{s}}--Sz.-Nagy theorem as a
corollary of the above proposition. 
\begin{cor} 
Let $V\in\LLL(H)$ for a Hilbert space $H$. Then $V$ is the cogenerator
of a contractive $C_0$-semigroup if and only if $V$ is contractive and
$1\notin P_\sigma(V)$.
%be a contractive operator on a Hilbert space $H$. If $1\notin
%P_\sigma(V)$, then $V$ is the cogenerator of a contractive
%$C_0$-semigroup. 
\end{cor}
\begin{proof} Assume that $V$ is contractive and $1\notin
  P_\sigma(V)$. 
By Lemma \ref{lemma:resolvent}, the operator
  $A:=(V+I)(V-I)^{-1}$ satisfies $(0,\infty)\subset
  \rho(A)$. Moreover, it is densely defined by the mean ergodic
  theorem, see Yosida \cite[Theorem VIII.3.2]{yosida-book}. 

By Proposition \ref{prop:contr-sgrs-via-v-tau} it is
  enough to show that contractivity of $V=V_1$ implies that all operators $V_\tau$, $\tau>0$, are contractive. Take
  $\tau>0$ and $x\in X$. For  $t:=\frac{1}{\tau}$ and
  $y:=-R(t,A)x=(A-tI)^{-1}x$ we have
\begin{eqnarray}
\|V_\tau x\|^2-\|x\|^2&=&\|(A+tI)(A-tI)^{-1}x\|^2 -  \|x\|^2 \label{eqnarray:v_tau-hilbert} \\
&=&\langle (A+tI)y, (A+tI)y \rangle - \langle (A-tI)y, (A-tI)y \rangle =
4t \Re \langle Ay,y \rangle, \nonumber
  \end{eqnarray}  
so contractivity of $V_\tau$ is independent of $\tau$. 

Conversely, if $V$ is the cogenerator of a contractive
$C_0$-semigroups with generator $A$, then contractivity of $V$ follows
from (\ref{eqnarray:v_tau-hilbert}) for $\tau=1$ and the Lumer--Phillips
theorem. Moreover, by $I-V=2R(1,A)$ we have $1\notin P_\sigma(V)$. 
\end{proof}
\begin{remark}
As we see from the above proof, the following nice property holds for
cogenerators of $C_0$-semigroups on Hilbert spaces:
Contractivity of $V$ automatically implies contractivity of every
$V_\tau$, $\tau>0$. As we will see in the next sections, on Banach
spaces this property fails in general.  
\end{remark}

A result analogous to Proposition \ref{prop:contr-sgrs-via-v-tau} 
holds for generators of bounded $C_0$-semigroups as well.  
\begin{prop}\label{prop:bdd-sgrs-via-v-tau}
Let $A$ be a densely defined operator on a Banach space $X$. Then the following
assertions are equivalent.
\begin{enumerate}[(i)]
\item $A$ generates a $C_0$-semigroup $(T(t))_{t\geq 0}$ satisfying
  $\|T(t)\|\leq M$ for every $t\geq 0$.

\item $(0,\infty)\subset\rho(A)$ and the operators $V_\tau$ satisfy 
   \begin{equation*} 
   \left\|\left[\frac{V_\tau-I}{2}\right]^{n}\right\|\leq M \quad \text{for all }
   \tau>0 \text{ and } n\in\N.
     \end{equation*}

\item There exists $\tau_0$ such that
  $(\frac{1}{\tau_0},\infty)\subset\rho(A)$ and the operators $V_{\tau}$ satisfy  
   \begin{equation*} 
   \left\|\left[\frac{V_\tau-I}{2}\right]^{n}\right\|\leq M \quad \text{for all }
   0<\tau<\tau_0 \text{ and } n\in\N.
     \end{equation*}
\end{enumerate}
\end{prop}
\noindent The proof follows from  formula (\ref{eq:resolvent-via-v-tau})
and the Hille--Yosida theorem for bounded semigroups. 

Propositions \ref{prop:contr-sgrs-via-v-tau} and
\ref{prop:bdd-sgrs-via-v-tau} imply the following sufficient 
condition on Banach spaces being analogous to the one of Foia{\c{s}} and Sz.-Nagy.  
\begin{thm}\label{thm:growth-via-v-tau}
Let $A$ be densely defined on a Banach space $X$.
%  such that the operators $V_\tau$ exist for all $\tau\in
%  (0,\tau_0)$. 
Then the following assertions hold.
\begin{enumerate}[(a)]
\item If there exists $\tau_0>0$ such that
  $(\frac{1}{\tau_0},\infty)\subset \rho(A)$ and the operators
  $V_\tau$ are contractive for every $\tau\in(0,\tau_0)$, then $A$
  generates a contractive $C_0$-semigroup.
\item If there exists $\tau_0>0$ such that
  $(\frac{1}{\tau_0},\infty)\subset \rho(A)$ and the operators
  $V_\tau$ satisfy $\|V_\tau^n\|\leq M$ for all $\tau\in(0,\tau_0)$
  and $n\in\N$, then $A$ generates a $C_0$-semigroup $(T(t))_{t\geq
    0}$ with $\|T(t)\|\leq M$ for all $t\geq 0$.
\end{enumerate}
\end{thm}
\begin{proof}
Assertion (a) follows immediately from Proposition
\ref{prop:contr-sgrs-via-v-tau}. To prove (b) assume that
$\|V_\tau^n\|\leq M$. Then we have
\begin{equation*}
\left\|\left[\frac{V_\tau-I}{2}\right]^{n}\right\|\leq \frac{1}{2^n}
\sum_{j=0}^n 
\left({ 
%\small 
\begin{matrix} n  \\ j \end{matrix}} \right)
\|V_\tau^j\| \leq \frac{M\cdot 2^n}{2^n}=M
\end{equation*}
and (b) follows from Proposition \ref{prop:bdd-sgrs-via-v-tau}. 
\end{proof}
\begin{remark}\label{remarks:sgr-via-v-tau}  
In Proposition
\ref{prop:contr-sgrs-via-v-tau}, Proposition
\ref{prop:bdd-sgrs-via-v-tau} and Theorem \ref{thm:growth-via-v-tau} 
it suffices to consider $\{V_{\tau_n}\}_{n=1}^\infty$ for a sequence
$\tau_n>0$ converging to zero.
This again follows from the fact that in the Hille--Yosida theorem it
suffices to check the resolvent condition only for a sequence
$\{\lambda_n\}_{n=1}^\infty\subset (0,\infty)$ converging to infinity, which follows directly
from its proof.
\end{remark}

We finish this section by the following observation.  If $V$ is
contractive or power bounded, then so is the operator $-V$. Note that
$-V$ is the cogenerator of the semigroup generated by $A^{-1}$ if
$A^{-1}$ generates a $C_0$-semigroup. However, contractivity or
boundedness of $(e^{tA})_{t\geq 0}$ does not imply the same property
of $(e^{tA^{-1}})_{t\geq 0}$ (see Zwart \cite{zwart:2007-BCP} and also
Section \ref{section:examples} for elementary examples).
We refer to Zwart \cite{zwart:2007-BCP, zwart:2007-SF}, 
Gomilko, Zwart \cite{gomilko/zwart:2007}, Gomilko, Zwart, Tomilov \cite{gomilko/zwart/tomilov}, de Laubenfels
\cite{delaubenfels:1988} for further information on this aspect.
%the operator
%$A^{-1}$ for a generator $A$ and the related semigroup. 

Moreover, we have the following relation.
\begin{remark}
Assume that $(0,\infty)\subset \rho (A)$ and $A^{-1}$ exists as a
densely defined operator.
Then we have
\begin{equation}
V_{\tau,A^{-1}}=(\tau A^{-1}+I)(\tau A^{-1}-I)^{-1}=(\tau I +A)(\tau I-A)^{-1}=-V_\frac{1}{\tau}.
\end{equation}
So we see that contractivity (uniform power boundedness) of $V_\tau$
for \textit{all} $\tau>0$ or even for some sequences $\tau_{n,1}\to 0$ and $\tau_{n,2}\to \infty$
implies that $A$ and $A^{-1}$ both generate a contractive (bounded)
$C_0$-semigroup. 
%This fact we will use in Section \ref{section:examples} to construct some counterexamples. 
\end{remark}
Conversely, Gomilko \cite{gomilko:2004} and Guo,
Zwart \cite{guo/zwart:2006} showed for Hilbert spaces that if $A$ and
$A^{-1}$ both generate bounded semigroups, then the cogenerator $V$ is
power bounded (and hence so are all operators $V_\tau$ by the rescaling argument).  

It is an interesting and open question whether contractivity (boundedness) of the
semigroups generated by $A$ and $A^{-1}$ implies contractivity (power
boundedness) of the cogenerator $V$ on Banach spaces.

%%%%%%%%%%%%%%%%%%%%%%%%%%%%%%%%%%%%%%%%%%%%%%%%%%%%%%%%%%%%%%%%%%%%%%%
%%                                                                   %%
%%                   EXAMPLES: CONTRACTIONS                          %%
%%                                                                   %%
%%%%%%%%%%%%%%%%%%%%%%%%%%%%%%%%%%%%%%%%%%%%%%%%%%%%%%%%%%%%%%%%%%%%%%%
\section{Examples}\label{section:examples}

%TE-begin
% Very recently, Piskarev, Zwart \cite{PiZw07} showed that the
% cogenerator of the (conrtactive) left shift semigroup on
% $C([0,\infty))$ is not contractive or power bounded and, more presicely, its powers
% grow like $\sqrt{n}$.  
%TE-end

In \cite{gomilko/zwart/tomilov} Gomilko, Zwart and Tomilov 
show that for every $p\in [1,\infty)$, $p\neq 2$, there exists a
contractive operator $V$ on $l^p$ such that $(V-I)^{-1}$ exists as a
densely defined operator, but $V$ is not a cogenerator of a $C_0$-semigroup. 

The idea of their construction is the following. One considers the generator $A:=S_l-I$ for the left
shift $S_l$ given by $S_l(x_1,x_2,x_3,\ldots)=(x_2,x_3,\ldots)$.  The
corresponding cogenerator $V=S_lR(2,S_l)$ is a contraction by
contractivity of $S_l$ and the Neumann series for the
resolvent. Further, one shows that $A^{-1}$ does not generate a
$C_0$-semigroup which is the hard part. As a consequence one obtains
that the contraction $-V$ is not a cogenerator of a $C_0$-semigroup.  

Komatsu \cite[pp. 343--344]{komatsu:1966} showed that the operator
$A:=S_r-I$ for the right shift $S_r$ given by $S_r(x_1,x_2,x_3,\ldots)=(0,x_1,x_2,\ldots)$ 
on $c_0$ satisfies the same properties, i.e.,
$A^{-1}$ does not generate a $C_0$-semigroup. Since the
cogenerator $V$ corresponding to $A$ is contractive as well,
we have a contraction on $c_0$ which is not a cogenerator of a 
$C_0$-semigroup. 
%By going to the dual space $l^1$ we have an analogous example there.
%
% PROOF FOR US: $V'$ is contractive as well. Assume that $A'$ and $(A^{-1})'$ both generate $C_0$-semigroup. (Else we have nothing to prove.) Then $(A^{-1})''$ generate a weak*-continuous semigroup on $l^\infty$. Since the restriction of $(A^{-1})''$ on $c_0$ is $A^{-1}$, we se that it generates a weakly continuous semigroup on $c_0$. Since every weakly continuous semigroup is strongly continuous, the argument is finished.

%See Zwart \cite{zwart:2007-BCP} for analogous
%example in $C_0[0,1]$??? not clear!!! (Was the cogenerator there contractive???)

The following example shows that even for $X=\C^2$ the semigroup
cogenerated by a contraction need not to be
contractive. Note that the cogeneration property is no problem here.

In particular, this example 
and Example \ref{ex:contr-sgr-noncontr-cogen} show 
that none of the implications in the
Foia{\c{s}}--Sz.-Nagy theorem holds 
even on two-dimensional Banach spaces. 
%{\color{red} Tanja here I move the first pragraph is this section}
%
\begin{example}\label{ex:matrix-2x2}
Take $X=\C^2$ considered with $\|\cdot\|_p$, $p\neq 2$,
and
$A:=\left(\begin{smallmatrix} -1 & \beta \\  0 & -2 \end{smallmatrix}\right)$
for $\beta >0$.
The semigroup generated by $A$ is 
\begin{equation*}
T(t)=\left( 
\begin{matrix}
  e^{-t} & \beta (e^{-t} - e^{-2t})\\
  0      & e^{-2t}     
  \end{matrix}
\right),\quad t\geq 0.
\end{equation*}

We first show that $(T(t))_{t\geq 0}$ is not contractive for
appropriate $\beta$. 
Consider first $p=\infty$ and $\beta>1$. 
We have $\|T(t)\|= (1+\beta)e^{-t}-\beta e^{-2t}=:f(t)$. Since
$f(0)=1$ and $f'(0)=\beta -1>0$, the semigroup is not contractive.

Let now $2<p<\infty$ and define $\beta :=(3^p-1)^\frac{1}{p}$. Then   
\begin{equation*}
\left\Vert T(t)\left({ 
%\small 
\begin{matrix} x  \\ 1 \end{matrix}} \right)\right\Vert_p^p= (e^{-t}x
    + \beta (e^{-t}-e^{-2t}))^p + e^{-2pt}=:f_x(t) \quad \text{for } x>0. 
 \end{equation*}
We have $f_x(0)=x^p + 1= \|(x,1)\|_p^p$.
%\Vert\left({ \small \begin{matrix} x  \\ 1 \end{matrix}} \right)\Vert_p^p$. 
Further, $f'_x(0)=px^{p-1}(\beta-x)-2p$, so the semigroup is not
contractive if $x^{p-1}(\beta-x)>2$ for some $x>0$. This is the case
for $x:=\frac{\beta}{2}$. Indeed, 
$x^{p-1}(\beta-x)=\left(\frac{\beta}{2}\right)^p=\frac{3^p-1}{2^p}>2$
for $p>2$.  

We now show that the cogenerator $V$ is contractive for $\beta\leq 3$
if $p=\infty$ and $\beta:=(3^p-1)^\frac{1}{p}$ if
$p\in(2,\infty)$. The cogenerator is given by 
\begin{equation*}
V=(I+A)(A-I)^{-1}=
\left(\begin{matrix} 0 & \beta \\  0 & -1 \end{matrix}\right)
\left(\begin{matrix} -\frac{1}{2} & -\frac{\beta}{6} \\  0 & -\frac{1}{3} \end{matrix}\right)
=\left(\begin{matrix} 0  &  -\frac{\beta}{3} \\ 0 &  \frac{1}{3}  \end{matrix}\right).
\end{equation*}
So for $p=\infty$ we have $\|V\|=\max\{\frac{1}{3},
\frac{\beta}{3}\}\leq 1$ for $\beta\leq 3$. 
For $p\in(2,\infty)$ we have $\|V\|^p=\|(-\frac{\beta}{3},
\frac{1}{3})\|_p^p=\frac{(\beta^p+1)}{3^p}\leq 1$ if and only if
$\beta\leq (3^p-1)^\frac{1}{p}$.

We see that for $p\in(2,\infty]$ there exists a contraction 
such that the cogenerated semigroup is not contractive. The analogous
assertion for $p\in[1,2)$ follows by duality. 
\end{example}
\begin{remark}
  \mbox{}From Theorem \ref{thm:growth-via-v-tau}, Remark \ref{remarks:sgr-via-v-tau} and the above example
  we see that there exist contractions $V$ (even on $\C^2$ with
  $l^p$-norm, $p\neq 2$) such that $V_\tau$ are not contractive for
  every $\tau$ in a small interval $(0,\tau_0)$.
% since the semigroup cogenerated by $V$ is not contractive.   
\end{remark}

\vspace{0.1cm}

%TEXT (other implication in Foias--Sz.-Nagy)

The following example gives a class of contractive semigroups with
non-contractive cogene\-rators and shows that such semigroups exist
even on $(\C^2,\|\cdot\|_\infty)$.  In particular, this provide a
two-dimensional counterexample to the converse implication in the
Foias--Sz.-Nagy theorem.
\begin{example}\label{ex:contr-sgr-noncontr-cogen}
Every operator $A$ generating a contractive
$C_0$-semigroup such that $A^{-1}$ generates a $C_0$-semigroup which
is not contractive leads to an example of a contractive semigroup with
non-contractive cogenerator. Indeed, by the previous remark, there
exists $\tau>0$ such that $V_\tau$ is not contractive. Therefore, the
operator $\tau A$ generates a contractive semigroup with
non-contractive cogenerator. 

For a concrete example consider $X:=\C^2$ endowed with $\|\cdot\|_\infty$
and $A$ as in Example \ref{ex:matrix-2x2}. Then
$(e^{tA})_{t\geq 0}$ is not contractive for $\beta>1$. 
%
%We saw that $-V$ is
%contractive if and only if $\beta \leq 3$. Note that $-V$ is the
%cogenerator of $e^{tA^{-1}}$. 
%
We show that the semigroup generated by $A^{-1}$ is contractive if and only if $\beta\leq 2$. 

Indeed, we have $A^{-1}=\left(\begin{smallmatrix} -1 & -\frac{\beta}{2} \\
    \ 0 & -\frac{1}{2} \end{smallmatrix}\right)$ and
\begin{equation*}
e^{tA^{-1}}=\left( 
\begin{matrix}
  e^{-t} & \beta (e^{-t} - e^{-\frac{t}{2}})\\
  0      & e^{-\frac{t}{2}}
  \end{matrix}
\right),\quad t\geq 0.
  \end{equation*}
Therefore
$\|e^{tA^{-1}}\|_\infty=\sup\{e^{-t}+\beta(e^{-\frac{t}{2}}-e^{-t}),e^{-\frac{t}{2}}\}$.
Hence $e^{tA^{-1}}$ is contractive if and only if
$g(t):=e^{-t}+\beta(e^{-\frac{t}{2}}-e^{-t})\leq 1$ for every $t>0$.  
We have $g(0)=1$ and
$g'(t)=-e^{-t}+\beta(e^{-t}-\frac{1}{2}e^{-\frac{t}{2}})
=e^{-t}[\beta(1-\frac{1}{2}e^{\frac{t}{2}})-1]$.
Since the function $t\to 1-\frac{1}{2}e^{\frac{t}{2}}$ is monotonically decreasing, we obtain that $g'(t)\leq 0$ for every
$t\geq 0$ is equivalent to $g'(0)\leq 0$, i.e., $\beta\leq 2$.

So we see that for $1<\beta\leq 2$ the semigroup generated by $A^{-1}$ is 
contractive while the semigroup generated by $A$ is not
contractive. The rescaling procedure described above leads to a
contractive semigroup (generated by $\tau A$ for some $\tau$) with
non-contractive cogenerator.   

Zwart \cite{zwart:2007-BCP} gives another example of an operator $A$ generating a contractive
$C_0$-semigroup such that the semigroup generated by $A^{-1}$ is not
contractive and even not bounded. He takes a nilpotent semigroup on $X=C_0[0,1]$ such that the semigroup generated
by $A^{-1}$ grows like $t^{1/4}$. By the rescaling procedure we again obtain a contractive semigroup with non-contractive cogenerator. 

%TO DO: Was the semigroup not contractive? %Other references ??? 

% Hans 2007: $X=C_0[0,1]$, $(T(t))$ nilpotent, but $(T_{-1}(t))$ grows
% like $t^\frac{1}{4}$.
\end{example}
\begin{remark} The above example for $2<\beta\leq 3$ yields a
  contractive cogenerator $V$ such that the semigroups generated by both operators $A$ and
  $A^{-1}$ are not contractive.  
This gives an example of a contraction $V$ on
$(\C^2,\|\cdot\|_\infty)$ such that operators $V_\tau$ are not
contractive for every $\tau\in(0,\tau_1)\cup(\tau_2,\infty)$,
$0<\tau_1<1<\tau_2$, by Remark \ref{remarks:sgr-via-v-tau}.
\end{remark}

\vspace{0.02cm}

%%%%%%%%%%%%%%%%%%%%%%%%%%%%%%%%%%%%%%%%%%%%%%%%%%%%%%%%%%%%%%%%%%%%%%%
%%                                                                   %%
%%                   POLYNOMIAL GROWTH                               %%
%%                                                                   %%
%%%%%%%%%%%%%%%%%%%%%%%%%%%%%%%%%%%%%%%%%%%%%%%%%%%%%%%%%%%%%%%%%%%%%%%

\section{Polynomial growth}\label{section:pol-bdd}

In this section we investigate the connection between polynomial growth of 
a $C_0$-semigroup and of its cogenerator.

We first recall that a $C_0$-semigroup $T(\cdot)$ (a bounded operator
$V$) is said to be of \textit{polynomial growth}\/ if $\|T(t)\|\leq
p(t)$ ($\|V^n\|\leq p(n)$) holds for some polynomial $p$ and every
$t\geq 0$ ($n\in\N$).

This property has been characterised via the resolvent of the
generator in Malejki \cite{malejki:2001}, Eisner \cite{eisner:2005},
Eisner, Zwart \cite{eisner/zwart:2007}. See also Gomilko
\cite{gomilko:1999}, Shi and Feng \cite{shi/feng:2000} for the
boundedness case.

The following result shows that polynomial growth of a
$C_0$-semigroup on a Banach space implies polynomial growth of
its cogenerator.  This generalises a result of Hersch and Kato
\cite{HeKa79} and Brenner and Thom\'ee \cite{BrTh79} on bounded
semigroups. For the proof of this result, we need the following
estimate.
\begin{lemma}
  \label{la:Laguerre} Let $L^1_n(t)$ denote the first generalised
  Laguerre polynomial, i.e.,
  \begin{equation}
    \label{eq:5.2}
      L_n^1(t) = \sum_{m=0}^n \frac{(-1)^m}{m!} \left(\begin{array}{c} n+1\\ n-m \end{array} \right) t^m.
  \end{equation} 
  Then we have for fixed $k \in {\mathbb N}$ that 
  \begin{equation}
    \label{eq:5.0}
     %\int_0^{\infty} |L_n^1(2t)| e^{-t} t^k dt = O(n^k\sqrt{n}).
     C_1 n^k\sqrt{n}\leq \int_0^{\infty} |L_n^1(2t)| e^{-t} t^k dt \leq C_2 n^k\sqrt{n}
  \end{equation}
  for some constants $C_1, C_2$ and all $n\in \N$.
\end{lemma}
\begin{proof}
  We begin by showing that the integral can be bounded from below by a
  constant times $n^k \sqrt{n}$. For this we use the idea in
  \cite{BrTW75}.
  
  The Laplace transform of $g(t):=-2 L_{n-1}^1(2t) e^{-t}$ equals
  \[
  G(s) = \left(\frac{s-1}{s+1}\right)^n -1.
  \]
  Thus the Laplace transform of $f(t):=g(t)t^k$, $k\geq 1$, is given
  by
  \begin{equation}
  \label{eq:5.5}
    F(s)= (-1)^k \frac{d^k}{ds^k} G(s) = \sum_{m=1}^k
    \left(\frac{s-1}{s+1}\right)^{n-m} P_{n,m}\left(\frac{1}{s+1}\right)
  \end{equation}
  where $P_{n,m}$ are polynomials with order $\leq 2^k$.
  
  Since $\frac{s-1}{s+1}$ has norm one on the imaginary axis, we can
  write
  $F_m(s):=\left(\frac{s-1}{s+1}\right)^{n-m}\frac{1}{(s+1)^{\ell_m}}$
  on the imaginary axis as $\frac{1}{(i\omega +1)^{\ell_m}}
  e^{i(n-m)\Phi(\omega)}$ for some real-valued function $\Phi$.
  Furthermore, $\Phi''(\omega)$ is non-zero for almost all $\omega$.
  Note that $\Phi(\omega)$ equals $-2\arctan(\omega)+\pi$. From
  Corollary 1.5.1 of \cite{BrTW75}, we know that the induced
  multiplier norm\footnote{In the formulation of this corollary it is
    assumed that the function in front of $e^{i(n-m)\Phi(\omega)}$ has
    compact support. However, this is not used in the proof of this
    corollary.  In the proof it is assumed that there exists a
    $C^{\infty}$ function, $\chi$, with compact support such that
    $\chi$ divided by the function (in front of the exponential) has
    compact support, lies in $C^{\infty}$, and $\Phi''$ has no zeros
    in the support of $\chi$} of $\frac{1}{(s+1)^{\ell_m}}
  \left(\frac{s-1}{s+1}\right)^{n-m}$ on $L^{\infty}$ is larger or
  equal to $c_m \sqrt{n}$ for some constant $c_m$. This induced norm
  equals the $L^1(0,\infty)$-norm of the function $f_m(t)$, i.e., the
  inverse Laplace transform of $F_m(s)$.  The constant
  $C_{n,k}=\frac{n!}{(n-k)!}  \alpha_k$ appearing in $P_{n,k}$ is the
  one with the highest power of $n$. Thus for $n$ large, the
  $L^1(0,\infty)$-norm of $f$ behaves like the $L^1(0,\infty)$-norm of
  $C_{n,k} f_k(t)$, which is larger or equal to $C_1(n^k \sqrt{n})$.

  This proves that there exists a lower bound for the
  $L^1(0,\infty)$-norm of $f$ which is of the order $n^{k}\sqrt{n}$.
  To prove the upper bound, we use the Carlson estimate, see \cite{BrTh79}:
  \begin{equation}
  \label{eq:5.6}
    \|f\|_1 \leq 2 \sqrt{\|f\|_2\|tf\|_2},
  \end{equation}
  where $\|\cdot\|_1$, $\|\cdot\|_2$ denote the $L^1(0,\infty)$-norm
  and $L^2(0,\infty)$-norm, respectively.
  For completeness we include the proof of this estimate. Take $c>0$
  and observe
  \begin{eqnarray*}
  \int_0^{\infty} | f(t) |dt 
  &=& \int_0^c|f(t)|dt + \int_c^{\infty} |f(t)|dt \\
  &\leq& \sqrt{c} \|f\|_2 + \sqrt{\int_c^{\infty} \frac{1}{t^2} dt 
  \int_{c}^{\infty} |tf(t)|^2 dt}%\\
  \leq \sqrt{c} \|f\|_2 + \sqrt{\frac{1}{c}} \|tf(t)\|_2,
  \end{eqnarray*}
  where we have used the Cauchy-Schwarz estimate twice.
  Choosing $c=\frac{\|t f\|_2}{\|f\|_2}$, we find (\ref{eq:5.6}).
 % (End of what can be deleted)
  
  So to obtain a $L^1(0,\infty)$-norm estimate of $f(t):=L_n^1(2t) e^{-t} t^k$ we must
  estimate the $L^2(0,\infty)$-norm of this function and of $t$
  times it.
  
  In order to estimate the $L^2(0,\infty)$-norms, we use Parseval
  identity for the Fourier transform, i.e., $\|f\|_2^2 =
  \frac{1}{2\pi} \|\hat{f}\|_2^2$. Furthermore, the Fourier transform
  of $f$ equals the Laplace transform of $f$ restricted to the
  imaginary axis. Finally, we have that the Fourier transform of
  $tf$ equals $i(\hat{f})'(\omega)$.

  Using (\ref{eq:5.5}), we see that 
  \begin{equation}
    \label{eq:5.7}
    \|\hat{f}\|_2 \leq \sum_{m=1}^{k} \left\| \left(\frac{i\omega-1}{i
    \omega+1}\right)^{n-m} P_{n,m}\left(\frac{1}{i\omega +1}\right)
   \right \|_2 = \sum_{m=1}^{k} \left\|P_{n,m}\left(\frac{1}{i\omega +1}\right) \right\|_2,
  \end{equation}
  where we have used that $\frac{i\omega-1}{i \omega+1}$ has absolute
  value one. So we must estimate the $L^2(-\infty,\infty)$-norm of
  $P_{n,m}\left(\frac{1}{i\cdot +1}\right)$. Since we are interested
  in the behaviour with respect to $n$, and since the order of this
  polynomial is independent of $n$, we may look at the coefficients.
  Again we have that $P_{n,k}$ has the coefficient $C_{n,k} =
  \frac{n!}{(n-k)!}\alpha_k$ with of the highest power of $n$. So
  combining this with (\ref{eq:5.7}) we see that
  \begin{equation}
    \label{eq:5.8}
    \|f\|_2 = \frac{1}{\sqrt{2\pi}}\|\hat{f}\|_2 \leq \gamma_1 n^k
  \end{equation}
  for some contant $\gamma_1$. 
  
  We further have that $F'(s)$ equals
  \[
    F'(s) = \sum_{m=1}^k \left(\frac{s-1}{s+1}\right)^{n-m-1}
  \frac{1}{(s+1)^2} \left[ 2(n-m) P_{n,m} \left(\frac{1}{s+1}\right) -
    \left(\frac{s-1}{s+1}\right)P_{n,m}'\left(\frac{1}{s+1}\right)
  \right].
  \]
  By a similar argument as above, we have that the $L^2$-norm on the
  imaginary axis of this function is bounded by a constant times
  $n^{k+1}$. Hence
  \begin{equation}
    \label{eq:5.9}
    \|t f\|_2 = \frac{1}{\sqrt{2\pi}} \|i(\hat{f})'\|_2\leq \gamma_2 n^{(k+1)}.
  \end{equation}
  Combining (\ref{eq:5.6}), (\ref{eq:5.8}) and (\ref{eq:5.9}) shows
  that $\|f\|_1 \leq C_2 n^{k} \sqrt{n}$.
\end{proof}
\begin{thm}
  \label{thm:5.1-new}
  Let $T(\cdot)$ be a $C_0$-semigroup on a Banach space with
  cogenerator $V$. If $\|T(t)\|\leq M(1+t^k)$ for some $M$ and
  $k\in\N\cup\{0\}$ and every $t\geq 0$, then $\|V^n\|\leq C_1
  n^{k+\frac{1}{2}}$ some $C_1$ and every $n\in\N$.

  Furthermore, this estimate cannot be improved, i.e., for every
  $k\in\N\cup\{0\}$ there exists a Banach space and a $C_0$-semigroup
  satisfying $\|T(t)\|=O(t^k)$, $t$ large, such that $\|V^n\|\geq C_2
  n^{k+\frac{1}{2}}$ for some $C_2>0$ and every $n\in\N$.
\end{thm}
\begin{proof}
The proof is based on the following relation between the semigroup and the cogenerator
\begin{equation}
  \label{eq:5.1}
  V^n h = h - 2 \int_0^{\infty} L_{n-1}^1(2t) e^{-t} T(t) h dt,
\end{equation}
where $L_n^1(t)$ is again the first generalised Laguerre polynomial, see
e.g.\ Gomilko \cite{gomilko:2004} or Butzer and Westpal \cite{BuWe70}.

Using the fact that the semigroup is of polynomial growth we find that
\begin{equation}
  \label{eq:5.3}
     \|V^n \| \leq 1 + 2 M \int_0^{\infty}\left| L_{n-1}^1(2t)\right| e^{-t} (1+t^k) dt.
\end{equation}
Now by Lemma \ref{la:Laguerre} we conclude that $\|V^n\| \leq C_1 n^{k+\frac{1}{2}}$ some $C_1$ and every $n\in\N$.

It remains to show that this estimate is sharp. Let 
$X_0:=C_0([0,\infty))$ be the Banach space of continuous functions on $[0,\infty)$ vanishing at
infinity, considered with the maximum-norm, $\|\cdot\|_{\infty}$. 
Let $(T_0(t))_{t\geq 0}$ be the left-shift semigroup on $X_0$, i.e.,
$\left(T_0(t)f\right)(\eta)= f(t+\eta)$. This is a strongly
continuous, contractive semigroup on $X_0$. 

As Banach space $X$ we take now $(k+1)$ copies of $X_0$, again with
the maximum norm $\|x\|_{X}= \max_{m=1,\ldots,k+1} \|x_m\|_{\infty}$. The
infinitesimal generator $A$ is given by
\[
  A = \left( \begin{array}{ccccc} A_0 & I & 0 & \cdots & 0\\ 0 & A_0 & I & & \vdots\\ \vdots && \ddots &\ddots & \\
  \vdots& && A_0 & I\\ 0 &\cdots& 0 & 0 &A_0 \end{array} \right),
\]
where $A_0$ is the infinitesimal generator of $T_0(t)$. The semigroup generated by $A$ is given by
\begin{equation}
\label{eq:5.4}
  T(t) = \left( \begin{array}{ccccc} T_0(t) & t T_0(t) & \frac{t^2}{2!} T_0(t) & \cdots & \frac{t^k}{k!} T_0(t)\\ 0 & T_0(t) & tT_0(t) & & \vdots\\ \vdots && \ddots &\ddots & \\
  \vdots& && T_0(t) & tT_0(t)\\ 0 &\cdots& 0 & 0 &T_0(t) \end{array} \right).
\end{equation}
Since $T_0(\cdot)$ is a contraction semigroup, it is easy to see that $\|T(t)\| \approx t^k$ for $t$ large. 

In order to give the idea of the further construction, we first choose
as $h$ in (\ref{eq:5.1}) the following function $h(\eta)=
\left(0,\cdots,0, \mathrm{sign}(L_{n-1}^1(2\eta))\right)^T$. Using the
definition of $T(\cdot)$, $T_0(\cdot)$ and (\ref{eq:5.1}), we have
\[
  \left(V^nh \right)(\eta) = \left(\begin{array}{c} 0\\ 0 \\ \vdots \\0\\  \mathrm{sign}(L_{n-1}^1(2\eta))\end{array}\right) - 2 \int_0^{\infty} L_{n-1}^1(2t) e^{-t}  \left(\begin{array}{c} \frac{t^k}{k!} \cdot \mathrm{sign}(L_{n-1}^1(2(t+\eta))\\ \frac{t^{k-1}}{(k-1)!} \cdot \mathrm{sign}(L_{n-1}^1(2(t+\eta))\\ \vdots \\ t \cdot \mathrm{sign}(L_{n-1}^1(2(t+\eta))\\ \mathrm{sign}(L_{n-1}^1(2(t+\eta)) \end{array}\right) dt.
\]
So
\begin{eqnarray}
  \nonumber
  \|V^nh\|_{X} &\geq& 
 %new
  %\|\left(V^nh\right)(0)\| 
  \|(V^n h)_1\|_{\infty} 
  \geq |(V^n h)_1(0)| =
 %new
  2 \int_0^{\infty} L_{n-1}^1(2t) e^{-t}  \frac{t^k}{k!} \cdot \mathrm{sign}(L_{n-1}^1(2t)) dt\\
  \label{eq:5.5-new}
   &=& \frac{2}{k!} \int_0^{\infty} |L_{n-1}^1(2t)| e^{-t} t^k dt.
\end{eqnarray}
Since by Lemma \ref{la:Laguerre} this last integral behaves like
$n^{k}\sqrt{n}$, we see that the estimate is sharp. Since $h$ is not
continuous and is not vanishing at infinity, we see that the above
construction is not finished. However, for every $\varepsilon>0$ one
can find a $h_{\varepsilon} \in X$ such that the equality in
(\ref{eq:5.5-new}) holds within an error margin of $\varepsilon$, see
Example 3.5 of \cite{zwart:2007-BCP}.
\end{proof}

\begin{remark}
  For Hilbert spaces one can obtain a sharper result. For $k=0$, i.e.,
  for bounded semigroups, Gomilko \cite{gomilko:2004} proved that on
  Hilbert spaces $\|V^n\|$ grows at most like $\ln(n+1)$. It is unknown
  whether this result is optimal.
\end{remark}

Our next result shows that for analytic semigroups the converse
implication in Theorem \ref{thm:5.1-new} holds, i.e., polynomial
growth of the cogenerator implies polynomial growth of the
semigroup.
%(Note that even on 
%Hilbert spaces there exist polynomially bounded operators $V$ with $1\notin P_\sigma(V)$ which are not cogenerators,
%see Gomilko, Zwart, Tomilov \cite{gomilko/zwart/tomilov}).  
%
\begin{thm}
Let $T(\cdot)$ be an analytic $C_0$-semigroup on a Banach space with
cogenerator $V$. 
If $\|V^n\|\leq C n^k$ for some $C,k\geq 0$ and every $n\in \N$, then
$\|T(t)\|\leq M(1+t^{2k+1})$ for some $M$ and every $t\geq 0$.
%In particular, $T(\cdot)$ is polynomially bounded if and only if $V$ is polynomially bounded.
\end{thm}
\begin{proof}
Assume $\|V^n\|\leq C n^k$ for some $C,k\geq 0$ and every $n\in \N$.
Then $r(V)\leq 1$ and therefore  $\lambda\in\rho(A)$ for 
$\Re\lambda>0$ by Lemma  \ref{lemma:resolvent}. 
Our aim is to show that there exist $a_0,M>0$ such that
\begin{eqnarray}
%\Re(\lambda)\|R(\lambda,A)\|
&\ &\|R(\lambda,A)\| \leq \frac{M}{(\Re\lambda)^{k+1}} \text{ for all }
\lambda \text{ with } 0<\Re\lambda<a_0, \label{eq:res-generator-power-bdd-1} \\
&\ &\|R(\lambda,A)\| \leq M  \text{ for all } \lambda \text{ with }
\Re\lambda\geq a_0. \label{eq:res-generator-power-bdd-2}
  \end{eqnarray} 
By Eisner, Zwart \cite[Thm. 2.1]{eisner/zwart:2007} this implies 
growth at most like $t^{2k+1}$ for analytic semigroups. 

Take some $a_0>\max\{0,\omega_0(T)\}$. Then
(\ref{eq:res-generator-power-bdd-2}) automatically holds and
we only have to show (\ref{eq:res-generator-power-bdd-1}).

Since $T(\cdot)$ is analytic, $R(\lambda,A)$ is uniformly bounded on 
$\{\lambda:|\Im \lambda|>b_0,\ 0<\Re\lambda<a_0\}$ for some $b_0\geq 0$. 
Moreover, $R(\lambda,A)$ is also uniformly bounded on 
$\{\lambda:|\Im \lambda|\leq b_0,\ \frac{1}{3}\leq \Re\lambda<a_0\}$
as well. 
Take now $\lambda$ with $0<\Re\lambda<\frac{1}{3}$ and
$-b_0<\Im\lambda<b_0$. 

By Lemma \ref{lemma:resolvent} we have 
\begin{equation}\label{eq:resolvent-gen-cogen}
\|R(\lambda,A)\|\leq \frac{1+\|V\|}{|\lambda-1|}\left\|R\left(\frac{\lambda+1}{\lambda-1},V\right)\right\|.
%R(\mu,V)=\frac{1}{\mu-1}-\frac{2}{(\mu-1)^2}R(\frac{\mu+1}{\mu-1},A).
\end{equation} 
By Eisner, Zwart \cite[Thm. 2.4]{eisner/zwart:2007}, growth of $\|V^n\|$ like $n^k$ implies 
\begin{eqnarray}\label{eq:pol-bdd-res-v}
\|R(\mu,V)\|\leq \frac{\tilde{C}}{(|\mu|-1)^{k+1}} \text{ for all } \mu \text{ with
} 1<|\mu|\leq 2
%&\ &\|R(\mu,V)\|\leq C \text{ for all } \mu \text{ with } |\mu|\geq 2 
  \end{eqnarray}
for some constant $\tilde{C}$. For $\mu:=\frac{\lambda+1}{\lambda-1}$
and $0<\Re\lambda<\frac{1}{3}$ we have  
\begin{equation*}
|\mu|-1= \frac{|\lambda+1|-|\lambda-1|}{|\lambda-1|}= \frac{4\Re
  \lambda}{|\lambda-1|(|\lambda+1|+|\lambda-1|)}
<1.
\end{equation*}
Then we use (\ref{eq:pol-bdd-res-v}) to obtain 
\begin{equation*}
\left\|R\left(\frac{\lambda+1}{\lambda-1},V\right)\right\|\leq
\frac{\tilde{C}|\lambda-1|^{k+1}(|\lambda+1|+|\lambda-1|)^{k+1}}{4^{k+1} (\Re\lambda)^{k+1}}
\leq \frac{C_1}{(\Re\lambda)^{k+1}}
  \end{equation*}
for $C_1:= \tilde{C}(b_0^2+\frac{5}{4})^{k+1}$. 
So, by (\ref{eq:resolvent-gen-cogen}),
\begin{equation*}
\|R(\lambda,A)\|\leq \frac{C_1(1+\|V\|)}{|\lambda-1|(\Re\lambda)^{k+1}}
\leq \frac{C_1(1+\|V\|)}{2(\Re\lambda)^{k+1}}
%\leq \frac{M}{(\Re\lambda)^k}    
  \end{equation*}
which proves (\ref{eq:res-generator-power-bdd-1}). 
\end{proof} 

%%%%%%%%%%%%%%%%%%%%%%%%%%%%%%%%%%%%%%%%%%%%%%%%%%%%%%%%%%%%%%%%%%%%%%%

\vspace{0.1cm}

\noindent {\bf Acknowledgement.} 
The authors are grateful to the referee for pointing out some missing steps in the first version of the paper.

\parindent0pt

\end{document}